\documentclass[12pt]{amsart}

\usepackage{amsmath,amsthm,amscd,amsfonts,amssymb,enumerate}
\usepackage{graphicx}
\usepackage{color}
\usepackage[colorlinks]{hyperref}
\usepackage{amsfonts,amssymb,amscd,amsmath,enumerate,url,verbatim}
 \usepackage[dvips]{epsfig}
 \usepackage[none]{hyphenat}
\usepackage{amsmath,amssymb,amsfonts,enumerate,amsthm}
 \usepackage{amsgen, amstext,amsbsy,amsopn, amsthm, amsfonts,amssymb,amscd,amsmat
 h,euscript,enumerate,url,verbatim,calc,xypic}
 \usepackage{latexsym}
 \usepackage{graphics}
 \usepackage{color}
\headsep=1cm \numberwithin{equation}{section}
\numberwithin{equation}{section}

\numberwithin{equation}{section}
\input xy
\xyoption{all}

\newtheorem{thm}{Theorem}[section]
\newtheorem{cor}[thm]{Corollary}

\newtheorem{lem}[thm]{Lemma}
\newtheorem{prop}[thm]{Proposition}
\newtheorem{defn}[thm]{Definition}

\newtheorem{exam}[thm]{Example}
\newtheorem{rem}[thm]{Remark}

\newcommand{\coker}{\operatorname{Coker}\,}
\newcommand{\Hom}{\operatorname{Hom}\,}
\newcommand{\Ext}{\operatorname{Ext}\,}

\newcommand{\Spec}{\operatorname{Spec}\,}

\newcommand{\Ass}{\operatorname{Ass}\,}

\newcommand{\Tr}{\operatorname{Tr}\,}

\newcommand{\N}{\mathbb{N}}

\newcommand{\fa}{\mathfrak{a}}
\newcommand{\fb}{\mathfrak{b}}

\newcommand{\fp}{\mathfrak{p}}

\begin{document}
%%% ----------------------------------------------------------------------
%%% ----------------------------------------------------------------------
%%% ----------------------------------------------------------------------
\bibliographystyle{amsplain}
%%% ----------------------------------------------------------------------
%%% ----------------------------------------------------------------------
%%% ----------------------------------------------------------------------

%%% ----------------------------------------------------------------------
%%% ----------------------------------------------------------------------
%%% ----------------------------------------------------------------------

\title[Linkage of sheaves of modules]
 {Linkage of sheaves of modules}

\address{Faculty of mathematical sciences and computer,
amirkabir university and iran national science foundation, tehran, iran.}

\email{frahmati@aut.ac.ir}

\email{sayyarikh@gmail.com}
%%% ----------------------------------------------------------------------
%%% ----------------------------------------------------------------------
%%% ----------------------------------------------------------------------
\bibliographystyle{amsplain}
%%% ----------------------------------------------------------------------
%%% ----------------------------------------------------------------------
%%% ----------------------------------------------------------------------

     \author[F. Rahmati]{Farhad Rahmati}
     \author[kh. sayyari]{khadijeh sayyari}



%$^* Corresponding author$

%%% ----------------------------------------------------------------------
%%% ----------------------------------------------------------------------
%%% ----------------------------------------------------------------------

\begin{abstract} Inspired by the works in linkage theory
of modules, we define the concept of linkage of sheaves of modules as a generalization of linkage of modules. Thus, we expressed it in geometry algebraic language.
%We used new functors $\overline{\lambda}$ and $\overline{\Tr}$ for this definition and studied their properties. They have shown that if $X$ is a connected scheme and $\mathfrak{F}$ be a coherent $\mathcal{O}_X$-modules then $\overline{Tr}(\overline{\Tr}\mathfrak{F})=\mathfrak{F}$. Also, they stated the concept of an $\mathcal{O}_X$-projective sheaf of modules and, with its help, they have proved that necessary and sufficient conditions for a coherent $\mathcal{O}_X$-modules $\mathfrak{F}$ be a locally free that  $\overline{\Tr}\mathfrak{F}=0$.
%By the defenition of linkage of sheaves of modules, if a sheaf of modules $\mathfrak{F}$ is linked then, for all open subset $U\subseteq X,$ the $\mathcal{O}_X(U)$-module $\mathfrak{F}(U)$ is syzygy. So, if is finitely generated and stable, then it is also linked module.
We show that the linkedness of sheaves is a locally property.
%if $X$ is a connected scheme and $\mathfrak{F}$ is a coherent $\mathcal{O}_X$-modules then $\mathfrak{F}$ is linked if and only if $\mathcal{O}_X\mid_U$-module $\mathfrak{F}\mid_U$ is linked, for every affine open subset $U\subseteq X.$ 
As an important result, we have shown that the sheaf of modules made of Glueing schemes and Glueing linked sheaves of modules is a linked sheaf.

%The authors were able to achieve a classification of linked sheaves, especially linked coherent sheaves. For example, they hav shown:

%The reseachers have defined the concept of a stable and syzygy sheaf of modules and shown a sheaf of coherent modules on a scheme is a linked if and only if it is stable and syzygy. 

Also, 
%in this paper,  
it has been shown that for every sheaf of modules on non-domain, it is possible to obtain a maximal linked subsheaf of modules. 
%More precisely, we have proven that
%Let $X$ be a connected noetherian scheme and $\mathfrak{F}$ be a coherent $\mathcal{O}_X$-modules.
%when $U\subseteq X$ is an affine open subset such that $\mathcal{O}(U)$ is not a domain and $\Ass\mathfrak{F}(U)\cap\Ass\mathcal{O}(U) \neq\varnothing $ then the set
%$$\sum = \{\mathfrak{F'}\mid  \mathfrak{F'}  \text{ is a linked subsheaf of } \mathfrak{F}\mid_ U\}$$
%is not empty and has maximal element. 
%Also, for all $\mathfrak{F'}$ and $\mathfrak{F''}$ belong to $max \sum$, $\mathfrak{F'} \bigoplus \mathfrak{F'}$ is not stable.
\end{abstract}

\maketitle

%%% ----------------------------------------------------------------------
%%% ----------------------------------------------------------------------
%%% ----------------------------------------------------------------------
\bibliographystyle{amsplain}
%%% ----------------------------------------------------------------------
%%% ----------------------------------------------------------------------
%%% ----------------------------------------------------------------------
\section{introduction}
  Classically, linkage theory refers to
Halphen (1870) and M. Noether \cite{No}(1882) who worked to classify
space curves. In 1974 the significant work of Peskine and Szpiro \cite{PS} brought
breakthrough to this theory and stated it in the modern algebraic
language; two proper ideals $\fa$ and $\fb$ in a Cohen-Macaulay local ring $R$ is said to be linked if
there is a regular sequence $\underline{x}$ in their intersection such that
$\fa =\underline{x}:_R \fb$ and $\fb = \underline{x} :_R \fa$. (benefit )Using of some tools like sheaves, duality and cohomology, they could answered to some important questions in regular local rings.

A new progress in the linkage theory is the work of Martsinkovsky
and Strooker
 \cite{MS} which established the concept of linkage of modules. 
 %They used the functors $\lambda$ and $\Tr$ for this definition. 
 %

% Let $R$ be a commutative Noetherian ring with $1\neq 0$ and $M$ be a finitely generated $R$-module.
  In this note, inspired by the works in the module case, we present the concept of linkage of sheaves of modules (for
convenience, we may write sheaves or sheaf respectively) on
  a scheme. Thus, we generalized this concept to algebraic geometry. For this purpose, we state the operation 
  %``$\overline{\lambda}$" and 
  ``$\overline{\Tr}$"  (see Definition \ref{B1}) and so this concept is as follows.
 % \begin{defn}\label{F1}
%2.1
%Let $X$ be a connected scheme and $\mathfrak{F}$ is a "stable" sheaf of $\mathcal{O}_X$-module such that ``$\mathfrak{F}$ has a free resolution". 
% $$\overset{l}\oplus\mathcal{O}_X  \overset{\phi}{\rightarrow} \overset{t}\oplus\mathcal{O}_X \overset{\varphi}{\rightarrow}\mathfrak{F} \rightarrow 0$$
%with $t, l \in \N.$ 
%Set $(-)^*:= \mathcal{H}om_{\mathcal{O}_X}(- ,\mathcal{O}_X),$ $\overline{\Tr}\mathfrak{F} := \coker \phi^*$ and $\overline{\lambda}\mathfrak{F} := \coker \varphi^*=\overline{\Omega}(\overline{\Tr}\mathfrak{F})$ where "$\overline{\Omega}$" is first syzygy.  Then, applying $(-)^*$, we get the exact sequences 
%\begin{equation}
%0 \rightarrow \mathfrak{F}^* \overset{\varphi^*}{\rightarrow} (\overset{t}\oplus\mathcal{O}_X)^*\rightarrow \overline{\lambda} \mathfrak{F} \rightarrow 0
%\end{equation}
%and
%\begin{equation}\label{e2}
%$$0 \rightarrow \mathfrak{F}^* \overset{\varphi^*}{\rightarrow} (\overset{t}\oplus\mathcal{O}_X)^* \overset{\phi^*}{\rightarrow} (\overset{l}\oplus\mathcal{O}_X)^* \rightarrow \overline{\Tr} \mathfrak{F} \rightarrow 0$$
%\end{equation}
%and
%\begin{equation}\label{e3}
%$$0 \rightarrow \overline{\lambda} \mathfrak{F} \overset{\phi^*}{\rightarrow} (\overset{l}\oplus\mathcal{O}_X)^* \rightarrow \overline{\Tr} \mathfrak{F} \rightarrow 0.$$
%\end{equation}
Let $X$ be a connected scheme and $\mathfrak{F}$ and $\mathfrak{G}$ be two ``stable" sheaves of $\mathcal{O}_X$-module such that they have free resolutions of finite rank;  we say 
$\mathfrak{F}$ and $\mathfrak{G}$ are linked if $\Omega\overline{\Tr}\mathfrak{F}\cong\mathfrak{G}$ and $\Omega\overline{\Tr}\mathfrak{G}\cong\mathfrak{F}.$
%\end{defn}

%This is a generalization of its classical concept, when $M = R$ (\cite{PS}).Also, all modules are finitely generated. 

It is citable concept that, in the case where $X=\Spec R$ is an affine scheme, linkedness of two stable finitely generated $R$-modules $M$ and $N$ implies linkedness
    of two sheaves of $\mathcal{O}_X$-modules $\overset {\sim}M$ and $\overset {\sim}N$ and vice versa (see Theorem \ref{C1}). 

 In this work, we consider the above generalization and study some of its basic facts. The organization of the paper goes as follows.

First, in Section 2, we present operations $\overline{\lambda}$ and $\overline{\Tr}$  and show that, among other things, they are functors on an affine scheme(Definition \ref{RR}). 
%Also, we study some their basic properties(Theorem \ref{B3}).

In Section 3, we consider
%some basic properties of 
these operations in the case where $\mathfrak{F}$ has free resolution. For instance, it is shown that 
%if $\mathfrak{F}$ is stable and has free resolution then 
$(\overline{\lambda}\mathfrak{F})\mid _U \cong \overline{\lambda}(\mathfrak{F}\mid_U)$ and $(\overline{\Tr}\mathfrak{F})\mid _U \cong \overline{\Tr}(\mathfrak{F}\mid_U)$(Theorem \ref{B5}).
Also, it is proven that a coherent sheaf $\mathfrak{F}$ is a locally free if and only if
$\overline{\Tr}\mathfrak{F}=0$(Proposition \ref{B4}).

In Section 4, we state the concept of linked sheaves.
In view of the definition
% of linkage of sheaves of modules 
it is natural to ask whether the sheaf obtained by glueing a family of linked sheaves is a linked sheaf. We prove that, in some special cases, it does.
More precisely, it is shown that if $(X, \mathfrak{F})$ is the sheaf obtained by glueing a family of sheaves $(X_i, \mathfrak{F}_i),$ under glueing sheaves and glueing schemes conditions, and suppose that, for each $i,$ the scheme $X_i$ is connected, then $\mathfrak{F}$ is linked if there exist integer numbers $t$ and $k$ such that $\{X_i, \mathfrak{F}_i\}$ are ``$(t , k)$ co-rank" linked sheaves of modules(Theorem \ref{C}). Conversely, it is shown that the linkedness of sheaves is a locally property(Theorem \ref{C2.3}).
  
Finally, it would be interesting to know whether there exists a linked subsheaf of an arbitrary sheaf. When $X$ is a Noetherian scheme, $\mathfrak{F}$ is a coherent sheaf and there exists an open affine subset $U\subseteq X$ such that $\mathcal{O}\mid_U$ is not an integral domain and $\Ass\mathfrak{F}(U)\cap\Ass\mathcal{O}(U) \neq\varnothing,$ we obtain that the set 
$$\sum = \{ \mathfrak{F'}\mid \mathfrak{F'}  \text{is a linked subsheaf of } \mathfrak{F}\mid_ U\}$$
is not empty and has maximal element(Theorem \ref{p.}).
  
Thorough out the text, we assume that $X$ is a connected scheme and $\mathfrak{F}$ is a sheaf of $\mathcal{O}_X$-module such that $\mathfrak{F}$ is a coherent sheaf or has a free
 resolution. 
Also, $R$ is a commutative Noetherian ring with $1\neq 0$ and all $R$-modules are finitely generated.

%++++++++++++++++++++++++++++++++++++++++++++++++++++++++++++++++++++++++++++++++++++++++++++++++++++

\section{Some functors over the category of sheaves of modules}

%-----------------------------------------------------------------------------------------------
%Thorough out this section, we suppose that $R$ is a commutative Noetherian ring with $1\neq 0$ and $M$ and $N$ are finitely generated $R$-module. Also, we assume that $X$ is a connected scheme and $\mathfrak{F}$ is an $\mathcal{O}_X$-module.

In this section, using free resolutions of a sheaf, we introduce the concept of operations $\overline{\Tr}$ and $\overline{\lambda}$ and we will show that they are functors over the category of coherent sheaves of modules on affine schemes. To this end, we recall some definitions.
 \begin{defn}\label{B1}
Let 
$\overset{t_2}\oplus\mathcal{O}_X  \overset{\phi}{\rightarrow} \overset{t_1}\oplus\mathcal{O}_X \overset{\varphi}{\rightarrow}\mathfrak{F} \rightarrow 0$
 be a free resolution of $\mathfrak{F}.$ 
% as $\mathcal{O}_X$-module. 
 So, we say $\mathfrak{F}$ has $(t_1 , t_2)$-ranks.
 
 By applying $(-)^*= \mathcal{H}om_{\mathcal{O}_X}(- ,\mathcal{O}_X),$ we get the exact sequence
\begin{equation}\label{e2}
0 \rightarrow \mathfrak{F}^* \overset{\varphi^*}{\rightarrow} (\overset{t_1}\oplus\mathcal{O}_X)^* \overset{\phi^*}{\rightarrow} (\overset{t_2}\oplus\mathcal{O}_X)^* \rightarrow \overline{\Tr} \mathfrak{F} \rightarrow 0
\end{equation}
where transpos of $\mathfrak{F},$ $\overline{\Tr}\mathfrak{F}$ denotes $\coker \phi^*.$ Similarly, one may consider $\overline{\lambda}\mathfrak{F} := \coker \varphi^*=\Omega(\overline{\Tr}\mathfrak{F})$ where "$\Omega$" is the first syzygy. Hence, we get other exact sequences
\begin{equation}
0 \rightarrow \mathfrak{F}^* \overset{\varphi^*}{\rightarrow} (\overset{t_1}\oplus\mathcal{O}_X)^*\rightarrow \overline{\lambda} \mathfrak{F} \rightarrow 0
\end{equation}
and
\begin{equation}\label{e3}
0 \rightarrow \overline{\lambda} \mathfrak{F} \overset{\phi^*}{\rightarrow} (\overset{l}\oplus\mathcal{O}_X)^* \rightarrow \overline{\Tr} \mathfrak{F} \rightarrow 0.
\end{equation}
\end{defn} 
 By this definition, the following remark is notable.
 \begin{rem}\label{B21}
 \begin{itemize}
 \item [(1)]
 By \cite[ exercise 2.1.2 and page 65]{H},
 $(\mathfrak{F})^*$,  $\overline{\Tr}\mathfrak{F}$ and $\overline{\lambda}\mathfrak{F}$ are sheaves of $\mathcal{O}_X$-module.

 \item [(2)]
If $\mathfrak{F}$ is a quasi coherent (or respectively coherent) then, by \cite[ proposition  2.5.7]{H}, $\overline{\Tr}\mathfrak{F}$ and $\overline{\lambda}\mathfrak{F}$ are quasi coherent (respectively coherent).
  \item [(3)]
If $\mathfrak{F}$ is free then, $\varphi$ is isomorphism and so $\overline{\Tr}\mathfrak{F}$ and $\overline{\lambda}\mathfrak{F}$ vanish.
\end{itemize}
\end{rem}
In view of the definition,
it is natural to ask ``whether any choice of free resolutions for $\mathfrak{F}$ has no effect and $\overline{\Tr} \mathfrak{F}$ may be uniquely defined ?". Theorem \ref{B24} shows that, in the case where $X$ is affine scheme, it does under ``stabilization". We will come back to this question again in Section 3.
\begin{defn}
%A finitely generated $R$-module 
It is said that $M$ is stable if $M$ has no free summands. Also, two $R$-modules $M$ and $N$ are stability isomorphic, denoted $\underline{M}=\underline{N},$ if there exist free modules of finite rank $H$ and $W$ with $H\oplus M\cong N\oplus W.$
% \end{defn}
Inspired by this definition, we construct the concept of a stable sheaf of modules.
%\begin{defn}
%Let $\mathfrak{F}$ be a $\mathcal{O}_X$-modules. 
We say $\mathfrak{F}$ is stable if there is no a split exact sequence  
$0 \rightarrow \overset{t}\oplus\mathcal{O}_X \rightarrow \mathfrak{F}.$ So, $\mathfrak{F}$ has no free direct summands.
Also, two sheaves of modules $\mathfrak{F}$ and $\mathfrak{G}$ are stability isomorphic and denote $\underline{\mathfrak{G}}=\underline{\mathfrak{F}} $ if there exist free sheaves of finite rank $\mathfrak{h}$ and $\mathfrak{w}$ with $\mathfrak{h}\oplus\mathfrak{G}\cong\mathfrak{F}\oplus\mathfrak{w}.$
\end{defn}
The following lemma shows that the stability of two modules is transferred to the stability of their associated sheaves and vice versa.
 \begin{lem}\label{B}
 Let $X=\Spec(R)$ be an affine scheme.
 % and $M$ and $N$ be finitely generated $R$-modules. 
 Then $M$ and $N$ are stability isomorphic if and only if $\overset {\sim} M$ and $\overset {\sim} N$ are stability isomorphic. In particular, $M$ is stable if and only if $\overset {\sim} M$ is so.
 \end{lem}
%The following lemma, among other things, shows that the concept of linkage of modules is related with the concept of linkage of modules in \cite{MS}. Indeed, it shows that $M$ and $N$ are linked if and only if $\overset {\sim} M$ and $\overset {\sim} N$ are linked.
Via the following lemma, one may define
$\overline{\Tr}$ and $\overline{\lambda}$ as functors over the category of coherent sheaves of $\mathcal{O}_X$-modules when $X$ is an offine scheme.
  \begin{lem}\label{B2}
Let $X=\Spec(R)$ be an affine scheme and 
%$M$ and $N$ be finitely generated $R$-modules. 
assume that free resolutions of 
%$\mathcal{O}_X$-modules 
$\overset{\sim}{M}$ and $\overset {\sim} N$ and a morphism $f : \overset{\sim}{M}\rightarrow\overset {\sim} N,$ as the following diagram, exist.
 $$\begin{CD}
         &&&&&&&&\\
         \ \ &&&& \oplus^{l_2}\mathcal{O}_X @>{h_1}>>\oplus^{l_1}\mathcal{O}_X@>{g_1}>>\overset{\sim}M @>>>0&  \\
          &&&&&&&& @ VV{f}V  \\
         \ \  &&&& \oplus^{l_4}\mathcal{O}_X @>{h_2}>>\oplus^{l_3}\mathcal{O}_X @>{g_2}>>\overset{\sim}N @>>>0.&\\
        \end{CD}$$\\
 Then       
        \begin{itemize}
        \item[(1)] applying ``$\Gamma (X, -)$" and '' $\sim$'', one may complete the above diagram
        %, by  in order, 
        and get the following commutative diagram
       %  Hence, we get 
 %  is abtained by applying '' $\sim$'' to the following diagram
%         the commutative diagram of $R$-modules is obtained.
% \begin{CD}
  %      &&&&&&&&\\ 
     %    \ \ &&&&  \oplus^{l_2}R  @>>> \oplus^{l_1}R@>>> M @>>>0&  \\
        %  &&&&  @VVV   @VVV  @VV{\Gamma (X,f)}V\\
%         \ \  &&&&   \oplus^{l_4}R @>>> \oplus^{l_3}R@>>>N @>>>0,&\\
   %     \end{CD}$$\\
%  for suitable integer numbers $l_i,$ and 
  $$\begin{CD}
         &&&&&&&&\\
         \ \ &&&& \oplus^{l_2}\mathcal{O}_X @>{h_1}>>\oplus^{l_1}\mathcal{O}_X@>{g_1}>>\overset{\sim}M @>>>0&  \\
          &&&&  @VVV   @VVV  @VV{f}V\\
         \ \  &&&&  \oplus^{l_4}\mathcal{O}_X @>{h_2}>>\oplus^{l_3}\mathcal{O}_X @>{g_2}>>\overset{\sim}N @>>>0.&\\
        \end{CD}$$\\
        \item[(2)] 
        %Applying $(-)^*:= \mathcal{H}om_{\mathcal{O}_X}(- ,\mathcal{O}_X)$ on the last diagram in previous item, 
        Also, in view of the definition \ref{B1}, one finds the following diagram
   $$\begin{CD}
         &&&&&&&&\\
         \ \ &&&& 0 @>>>\overset{\sim}{M^*} @>{g_1^*}>>(\oplus^{l_1}\mathcal{O}_X)^* @>{h_1^*}>>(\oplus^{l_2}\mathcal{O}_X)^*@>>>\overline{\Tr}\overset{\sim}{M}@>>>0&  \\
          &&&&&&  @VV{f^*}V   @VVV   @VVV  @VV{\widehat{f}^*}V\\
         \ \  &&&& 0 @>>> \overset{\sim}{N^*}@>{g_2^*}>>(\oplus^{l_3}\mathcal{O}_X)^*@>{h_2^*}>>(\oplus^{l_4}\mathcal{O}_X)^* @>>>\overline{\Tr}\overset{\sim}{N} @>>>0.&\\
        \end{CD}$$\\      
       % where $\mathfrak{Z} := \coker h_1^*$ and $\mathfrak{Z'} := \coker h_2^*.$
On the other hand, via \cite[2.6]{A}, there exists the following commutative diagram   
 $$\begin{CD}
         &&&&&&&&\\
         \ \ &&&& 0 \rightarrow\Ext^1_R(\Tr M, -) @>>> M\otimes_R - @>>>\Hom_R( M^+,-)@>>>\Ext^2_R(\Tr M, -)\rightarrow0&  \\
          &&&&  @VVV   @VVV   @VVV  @VVV\\
         \ \  &&&& 0 \rightarrow \Ext^1_R(\Tr N, -) @>>> N\otimes_R - @>>>\Hom_R( N^+,-)@>>>\Ext^2_R(\Tr N, -) \rightarrow0,&\\
        \end{CD}$$\\
where $(-)^+:= \Hom_R(-,R).$   
%         $Z_1= \Gamma(X,\mathfrak{Z}) = \coker \Gamma(X,h_1^*)^+$ and  $Z_2= \Gamma(X,\mathfrak{Z'}) = \coker \Gamma(X,h_2^*)^+.$
%So, by applying ``$\sim$'', we obtain
This diagram implies another commutative diagram
 $$\begin{CD}\label{b3}
         &&&&&&&&\\
         \ \ && 0\rightarrow\mathcal{E}xt^1_{\mathcal{O}_X}(\overline{\Tr}\overset{\sim}{M} , -) @>>>\overset{\sim}M\otimes_{\mathcal{O}_X} -@>>>\mathcal{H}om_{\mathcal{O}_X}(\overset{\sim}{M^*} ,- )@>>>\mathcal{E}xt^2_{\mathcal{O}_X}(\overline{\Tr}\overset{\sim}{M} , -) \rightarrow 0&  \\
      &&  @VV{\xi}V   @VV{f^*\otimes -}V @VV{\mathcal{H}om(f^*,-)}V  @VVV\\
         \ \  && 0\rightarrow\mathcal{E}xt^1_{\mathcal{O}_X}(\overline{\Tr}\overset{\sim}{N},-) @>>>\overset{\sim}N\otimes_{\mathcal{O}_X}-@>>>\mathcal{H}om_{\mathcal{O}_X}(\overset{\sim}{N^*} ,- )@>>>\mathcal{E}xt^2_{\mathcal{O}_X}(\overline{\Tr}\overset{\sim}{N},-)\rightarrow 0.&\\
        \end{CD}$$\\
        Now, by \cite[2.2]{A}, it
is straight forward to see that the morphism $\xi$ only depends on $f.$
 \end{itemize}
 \end{lem}
The next two corollary show that $\overline{\Tr}$ may be defined as a functor on the category of coherent sheaves on affine schemes under stabilization. In next section, we prove that is so in the case where $X$ is only a scheme.
\begin{cor}\label{B24}
Assume that the conditions of lemma \ref{B2} hold. Then $\underline{\overline{\Tr}\overset{\sim}{M}} \cong \underline{\overline{\Tr}\overset{\sim}{N}}$ if and only if $\underline{\overset{\sim}{M}} \cong \underline{\overset{\sim}{N}}$. In particular, any choice of free resolutions for $\overset{\sim}{M}$ has no effect and $\overline{\Tr}\overset{\sim}{M}$ is uniquely defined under stabilization.
\end{cor}
\begin{proof}
The result follows from the last commutative diagram in Lemma \ref{B2}.
%(\ref{b3}).\cite[2.3]{A}, \ref{B} and the fact that $\underline{\Tr M} \cong \underline{\Tr N}$ and $\underline{M} \cong \underline{N}$ are equivalent.
\end{proof}
\begin{cor}\label{RR}
%\begin{itemize}
%\item [(1)] Denote $\mathcal{M(R)}$ the category of finitely generated $R$-modules and $ \underline{\mathcal{M(R)}}$ the stabilizations of the categories $\mathcal{M(R)}.$ By \cite{A}, it is constructed a functor $\Tr : \underline{\mathcal{M(R)}}\rightarrow \underline{\mathcal{M(R)}}$ which formalizes this correspondence between $M$ and $Z.$ 
%the functor  on category of stabely modules which every finitely generated modules $M$ sends to $Z(M).$ Also, it is defined the functor $\lambda: = \Omega\Tr $ on this category.
%\item  [(2)]
Let $X$ be an affine scheme. Then 
%Denote $\mathcal{CSH}(X)$ and $ \underline{\mathcal{CSH}(X)}$  of the categories of coherent of $\mathcal{O}_X$-modules$\mathcal{O}_X$-modules . By \ref{B24} and \cite[2.5.5]{H}, we define a
 $\overline{\Tr}$ is a functor on the stabilizations of the category of coherent sheaf which sends $\mathfrak{F}$ to $\overline{\Tr}(\mathfrak{F}).$ In particular, $\overline{\lambda}:= \Omega \overline{\Tr}$ is a functor too. 
%the functor  on category of stabely modules which every finitely generated modules $M$ sends to $Z(M).$
 %Also, we define the functor $\lambda: = \Omega\Tr $ on this category.
% and next lemma shows that $\overline{\Tr}$ and $\overline{\lambda}$ are functors.
%\item [(3)]
% If $\mathfrak{F}$ is free, then $\overline{\Tr}\mathfrak{F}$ and $\overline{\lambda}\mathfrak{F}$ are % vanished.
%\end{itemize}
\end{cor}
 There are some relations between the transpose of a module and the transpose of its associated sheaf
as the following theorem shows.
 \begin{thm}\label{B3}
 Let $X=\Spec(R)$ be an affine scheme.
 % and $M$ be a finitely generated $R$-module. 
 Then the following statments hold.
\begin{enumerate}
\item 
 $\underline{\overline{\lambda} \overset {\sim} M} \cong \underline{ \overset {\sim} {\lambda M }}$
 
 \item
 $\underline{\overline{\Tr} \overset {\sim} M} \cong \underline{ \overset {\sim} {\Tr M}}$
 
 \item
 $\underline{\Gamma (X,\overline{\lambda} \overset {\sim} M)} \cong \underline{ \Gamma (X,\overset {\sim} {\lambda M })}=\underline{ \lambda M}$
 
 \item
 $\underline{\Gamma (X,\overline{\Tr} \overset {\sim} M)} \cong \underline{ \Gamma (X,\overset {\sim} {\Tr M })}=\underline{\Tr M}$
\item
$\underline{\overline{\Tr}(\overline{\Tr}\overset {\sim} M) }\cong \underline{ \overset {\sim} M}$.
\end{enumerate}
 \end{thm}
\begin{proof}
First note that using \ref{B}, one may assume that $M$ is stable. Let $F_2 \overset{f}{\rightarrow} F_1 \overset{g}{\rightarrow} M \rightarrow 0$ be a minimal free resolution of $M.$ Applying $(-)^+$, we get the exact sequences 
\begin{equation}\label{e1}
0 \rightarrow M^+ \overset{g^+}{\rightarrow} (F_1)^+ \rightarrow \lambda M \rightarrow 0
\end{equation}
%and
%\begin{equation}\label{e2}
%0 \rightarrow M^+ \overset{g^+}{\rightarrow} (F_1)^+ \overset{f^+}{\rightarrow} (F_2)^+ \rightarrow \Tr M \rightarrow 0
%\end{equation}
and
\begin{equation}\label{e30}
0 \rightarrow \lambda M \overset{f^+}{\rightarrow} (F_2)^+ \rightarrow \Tr M \rightarrow 0.
\end{equation}

 \begin{itemize}
 \item [(1)]
 %(or respectively (2)).
 Applying "$\sim$" on  (\ref{e1}), by \cite[propositions 2.5.5 and 2.5.2]{H}, we get the following exact sequenc of sheaves 
 %of $\mathcal{O}_X$-modules
\begin{equation}
0 \rightarrow \mathcal{H}om_{\mathcal{O}_X}(\overset{\sim} M,\mathcal{O}_X) \rightarrow \mathcal{H}om_{\mathcal{O}_X}(\overset{\sim} {F_1},\mathcal{O}_X) \rightarrow \overset{\sim}{\lambda M} \rightarrow 0 
\end{equation}
Now, the result follows using \ref{B24} and \ref{B1}.
\item [(2)] Follows using similar argument as used above, previous isomorphism and Five Lemma.
%Similarly, applying "$\sim$" on (\ref{e3}), 
%\begin{equation}
%0 \rightarrow \overset{\sim} {\lambda M} \rightarrow \mathcal{H}om_{\mathcal{O}_X}(\overset{\sim}{F_2},\mathcal{O}_X) \rightarrow \overset{\sim}{\Tr M} \rightarrow 0.
%\end{equation}
%(and (\ref{e3}) respectively)
\item [(3)]
and (4) Considering \cite[proposition 2.5.1] {H} and using items (1) and (2),
the result has desired. 
 \item [(5)]
 By \cite[2.6]{A}, $\Tr \Tr (-)\cong id (-)$ in the category of finitely generated modules. This, in conjunction
with Theorem \ref{B3}, implies that
$\overline{\Tr}(\overline{\Tr} \overset {\sim} M) \cong \overset {\sim} {M}.$
\end{itemize}
 \end{proof}
%-----------------------------------------------------------------------------------------------

 \section{Some properties of $\overline{\Tr}$ and $\overline{\lambda}$}
%Thorough out this section, we suppose that $X$ is a connected scheme and denote $\mathfrak{F}$ be an $\mathcal{O}_X$-module.
%In this section, we will show that the operations $\overline{\Tr}$ and $\overline{\lambda}$ are functors over the category of coherent sheaves on schemes and study their properties. To this end, we recall some definitions. 
  %\begin{defn}\label{B1}
%Let 
%\begin{equation}\label{e.}
%\overset{t_2}\oplus\mathcal{O}_X  \overset{\phi}{\rightarrow} \overset{t_1}\oplus\mathcal{O}_X \overset{\varphi}{\rightarrow}\mathfrak{F} \rightarrow 0 
%\end{equation}
% be a free resolution of $\mathfrak{F}.$ 
% as $\mathcal{O}_X$-module. 
% So, we say $\mathfrak{F}$ has the ranks $(t_1 , t_2),$ respectively. 
% \end{defn}
 
% such that $\mathfrak{F}$ has a free resolution 
%$$\overset{t_2}\oplus\mathcal{O}_X  \overset{\phi} {\rightarrow} \overset{t_1}\oplus\mathcal{O}_X \overset{\varphi}{\rightarrow}\mathfrak{F} \rightarrow 0$$
%
%Therefore, $\mathfrak{F}(U)$ is a finitely generated $\mathcal{O}(U)$-module for every open subset $U\subseteq X.$
%Thorough out this section, we suppose that $X$ is a connected topological space and denote $\mathfrak{F}$ be a $\mathcal{O}_X$-module.
The goal of this section is to study the oprations $\overline{\Tr}$ and $\overline{\lambda}$ over the category of sheaves of modules which has $(t_1 , t_2)$-rank, for some $t_i \in \N.$ 

First, note that any choice of free resolutions for $\mathfrak{F}$ has no effect and $\overline{\Tr}\mathfrak{F}$ and $\overline{\lambda}\mathfrak{F}$ are uniquely defined under stabilization. Indeed,
%we note the following remark.
 %\begin{rem}\label{R}
% $\mathfrak{F}\cong\mathfrak{F'}$ if and only if, for every$\fp \in X$,$\mathfrak{F}_{\fp}\cong\mathfrak{F'}_{\fp}.$ Indeed, in the case $\mathfrak{F}_{\fp}\cong\mathfrak{F'}_{\fp}$
  %to show $\mathfrak{F}\cong\mathfrak{F'}$ 
%  it is enough to consider the structure of associated sheaf $\mathfrak{F}^+$ of $\mathfrak{F}$ and show $\mathfrak{F}^+\cong\mathfrak{F'}^+.$
% \end{rem}
%\begin{rem}
let $\fp$ be an arbitrary point of a scheme $X.$ By the definition \ref{B1}, $\mathfrak{F}$ is generated by global sections and $\mathfrak{F}_{\fp}$ is a finitely generated $(\mathcal{O}_X)_{\fp}$-module (for convenience, we write $\mathcal{O}_{\fp}$-module). So, via \cite[2.5]{A}, $\Tr (\mathfrak{F}_{\fp})$ and $\lambda (\mathfrak{F}_{\fp})$ are unique under stabilization.
%Set $(-)^*:= \mathcal{H}om_{\mathcal{O}_X}(- ,\mathcal{O}_X)$ and $(-)^+:= \Hom_{\mathcal{O}_{\fp}}(-,\mathcal{O}_{\fp})$ and define
%$\mathfrak{Z} := \coker \phi^*$ and $Z := \coker (\phi_{\fp})^+.$ Applying $(-)^*$ on (\ref{e.}) and then 
Considering the stalks, one obtains a commutative diagram    
 $$\begin{CD}
         &&&&&&&&\\
         \ \ &&&& 0 @>>>(\mathfrak{F}^*)_{\fp} @>{(\varphi^*)_{\fp}}>>(\overset{t_1}\oplus\mathcal{O}_X^*)_{\fp} @>{(\phi^*)_{\fp}}>>(\overset{t_2}\oplus\mathcal{O}_X^*)_{\fp}@>>>(\overline{\Tr}\mathfrak{F})_{\fp}@>>>0&\\
          &&&&&&  @VV{\cong}V   @VV{\cong}V   @VV{\cong}V  @VV{\exists}V\\
           \ \  &&&& 0 @>>> (\mathfrak{F}_{\fp})^+@>>>\overset{t_1}\oplus(\mathcal{O}_{\fp})^+@>>>\overset{t_2}\oplus(\mathcal{O}_{\fp})^+  @>>>\Tr (\mathfrak{F}_{\fp}) @>>>0&\\
        \end{CD}$$\\
     and gets $(\overline{\Tr}\mathfrak{F})_{\fp}\cong \Tr(\mathfrak{F}_{\fp}).$
     
Now, suppose that $\overset{k_2}\oplus\mathcal{O}_X  \rightarrow \overset{k_1}\oplus\mathcal{O}_X \rightarrow\mathfrak{F} \rightarrow 0 $
is another free resolution of $\mathfrak{F}.$ So, $\underline{(\overline{\Tr}\mathfrak{F})_{\fp}}\cong \underline{(\overline{\Tr}'\mathfrak{F})_{\fp}}.$ This implies $\underline{\overline{\Tr}\mathfrak{F}}\cong \underline{\overline{\Tr}'\mathfrak{F}}.$ Also,
    $\underline{\overline{\lambda}\mathfrak{F}}= \underline{\overline{\lambda}' \mathfrak{F}}$ follows from using similar argument as used above.
       
%\end{rem}

%In the case $\mathfrak{F}$ and $\mathfrak{G}$ are stability isomorphism $\underline{\overline{\Tr}\mathfrak{F}}\cong \underline{\overline{\Tr}\mathfrak{G}}$ and $\underline{\overline{\lambda}\mathfrak{F}} \cong \underline{\overline{\lambda}\mathfrak{G}}.$
% \end{itemize}
% \end{rem}
% By \cite{A}[2.6], the operators $\Tr$ and  $\lambda$ are functors on category of "stable" modules.
%We state some properties of these fanctors in the following lemma.
 
 Next, we examine the effects of the above operations on the resterictions of a sheaf of modules under an open subsets of a scheme. To this end, the following two items can be useful.
 
 \begin{lem}\label{R..1}
%Let $\mathfrak{F}$, $\mathfrak{F'}$ and $\mathfrak{F''}$ be . 
The sequence 
%\begin{equation}{\label{e4}}
$$0\rightarrow\mathfrak{F}\rightarrow\mathfrak{F'}\rightarrow\mathfrak{F''}\rightarrow 0$$
%\end{equation}
of $\mathcal{O}_X$-modules is exact if and only if, for each open subset $U\subseteq X$, the sequence 
% \begin{equation}{\label{e5}}
$$0\rightarrow\mathfrak{F}\mid_U\rightarrow\mathfrak{F'}\mid_U\rightarrow\mathfrak{F''}\mid_U\rightarrow 0$$
%\end{equation}
 is exact. In particular, $\mathfrak{F}\cong\mathfrak{F'}$ if and only if for every open subset $U\subseteq X$, $\mathfrak{F}\mid_U\cong\mathfrak{F'}\mid_U.$ 
 \end{lem}
 %\begin{proof}
%It follows from calculating the stalks of (\ref{e4}) and using \cite[exercise 2.1.2]{H} and the fact that $(\mathfrak{F}\mid_U)_{\fp}=\mathfrak{F}_{\fp}$ for every $\fp\in U.$
%First, assume that $U\subseteq X$ is open subset, $\fp\in U.$ Calculating the stalks of (\ref{e4}), by \cite[exercise 2.1.2]{H}, we get
%$0\rightarrow\mathfrak{F}_{\fp}\rightarrow\mathfrak{F'}_{\fp}\rightarrow\mathfrak{F''}_{\fp}\rightarrow 0.$ The claim follows also from \cite[exercise 2.1.2]{H} and the fact that $(\mathfrak{F}\mid_U)_{\fp}=\mathfrak{F}_{\fp}.$

%Conversely, asssume that $V\subseteq U$ is open subset. Applying $\Gamma (V,-)$ on (\ref{e5}), by \cite[exercise 2.1.8]{H}, we get $0\rightarrow\mathfrak{F}(V)\rightarrow\mathfrak{F'}(V)\rightarrow\mathfrak{F''}(V).$ Therefore, there are morphisms $0\rightarrow\mathfrak{F}\rightarrow\mathfrak{F'}$ and $\mathfrak{F'}\rightarrow\mathfrak{F''}.$ Simirarly, by \cite[exercise 2.1.2]{H}, we get
%$0\rightarrow\mathfrak{F}_{\fp}\rightarrow\mathfrak{F'}_{\fp}\rightarrow\mathfrak{F''}_{\fp}\rightarrow 0.$ Exactness follows from \cite[exercise 2.1.2]{H}.
 
 %\end{proof}
% \begin{cor}
%Let $\mathfrak{F}$ and $\mathfrak{F'}$ be $\mathcal{O}_X$-modules such that Then 
% \end{cor}
 \begin{cor}\label{R.2}
 Let $\mathfrak{F}$ and $\mathfrak{F'}$ be coherent $\mathcal{O}_X$-modules such that, for every affine subset $U\subseteq X$, $\mathfrak{F}\mid_U\cong\mathfrak{F'}\mid_U.$ Then $\mathfrak{F}\cong\mathfrak{F'}.$  
 %Indeed, $ X\subseteq\cup U_i $ for some affine sets $U_i.$ Set $\mathfrak{F}_i=\mathfrak{F}\mid_{U_i}$ and $\mathfrak{F'}_i=\mathfrak{F'}\mid_{U_i}.$ 
% So, $\mathfrak{F}_i\mid_{U_i\cap U_j}=\mathfrak{F}_j\mid_{U_i\cap U_j}$ and
%  $\mathfrak{F'}_i\mid_{U_i\cap U_j}=\mathfrak{F'}_j\mid_{U_i\cap U_j}.$ The result follows by Glueing of sheaves and uniqeness.
 \end{cor}
% The following theorem shows that, in the case $\mathfrak{F}$ is a coherent sheaf an a scheme, any choice of free resolutions for F has no effect and TrF may be uniquely defined

 In the following remark, we compare free resolutions of a sheaf and its restriction to open subsets.
 \begin{rem}\label{r3}
As $X$ is connected, the rank of a locally free sheaf is the same everywhere. Therefore, 
 for any $U \subseteq X,$ 
$(\overset{t}\oplus\mathcal{O}_X)\mid_U= \overset{t}\oplus(\mathcal{O}_X\mid_U)$ 
and so, if $\mathfrak{F}$ has $(t_1 , t_2)$-ranks,
$\mathfrak{F}\mid_U$ also has $(t_1 , t_2)$-ranks.
 \end{rem}
 %\begin{lem}\label{B7}
%Let $\mathfrak{F}$ be a coherent $\mathcal{O}_X$-modules. Then $\mathfrak{F}$ has fee resolution.
%\end{lem}
%By the assumption, $X$ can be covered by open affine subsets $ U_i = \Spec R_i$ such that, for each $i$, there exists a finitely generated $R_i$-module $M_i$ with $\mathfrak{F}\mid_{U_i}=\overset{\sim} M_i.$
%\end{proof}
The following theorem studies the resterictions of $\overline{\Tr}\mathfrak{F}$ to open subsets 
 in the case where $X$ is a scheme. It will be used in the next Theorem which consider the functorness of $\overline{\Tr}$.
%The following theorem studies the relations between these operators and their .
\begin{thm}\label{B5}
Assume that $\mathfrak{F}$ has $(t_1 , t_2)$-ranks and $U$ is an open subset of $X$. Then $\underline{(\overline{\lambda}\mathfrak{F})}\mid _U \cong\underline{ \overline{\lambda}(\mathfrak{F}\mid_U)}$ and $\underline{(\overline{\Tr}\mathfrak{F})}\mid _U \cong \underline{\overline{\Tr}(\mathfrak{F}\mid_U)}.$
  \end{thm}
  \begin{proof}
 % Consider the sequence in \ref{B1}. 
We can assume that $\mathfrak{F}$ is stable. By the assumption and \ref{r3}, 
  $\overset{t_2}\oplus\mathcal{O}_X\mid_U  \rightarrow \overset{t_1}\oplus\mathcal{O}_X\mid_U \rightarrow\mathfrak{F}\mid_U\rightarrow 0.$ Applying $(-)^*:= \mathcal{H}om_{\mathcal{O}_X\mid_U}(- ,\mathcal{O}_X\mid_U),$ we get
\begin{equation}\label{e'..}
0 \rightarrow (\mathfrak{F}\mid_U)^* \rightarrow (\overset{t_1}\oplus\mathcal{O}_X\mid_U)^* \rightarrow \overline{\lambda} (\mathfrak{F}\mid_U) \rightarrow 0.
\end{equation}  
Now, it is enough to show that $((\overline{\lambda}\mathfrak{F})\mid _U)_{\fp} \cong (\overline{\lambda}(\mathfrak{F}\mid_U))_{\fp},$ for each point $\fp \in U.$ Let $\mathfrak{h}$ and $\mathfrak{G}$ are $\mathcal{O}_{X}$-sheaves of module.The structure of the stalks implies
 \begin{eqnarray*}
 (\mathcal{H}om_{\mathcal{O}_X\mid_U}(\mathfrak{G}\mid_U ,\mathfrak{h}\mid_U))_{\fp}& =& lim_{\fp\in V \subseteq U} \mathcal{H}om_{\mathcal{O}_X\mid_U}(\mathfrak{G}\mid_U ,\mathfrak{h}\mid_U)(V)= lim_{\fp\in V} \Hom_{\mathcal{O}_X\mid_V}(\mathfrak{G}\mid_U\mid_V ,\mathfrak{h}\mid_U\mid_V)\\ &=& lim_{\fp\in V} \Hom_{\mathcal{O}_X\mid_V}(\mathfrak{G}\mid_V ,\mathfrak{h}\mid_V) =  (\mathcal{H}om_{\mathcal{O}_X}(\mathfrak{G} ,\mathfrak{h}))_{\fp} .
 \end{eqnarray*}
 So, via \cite[exercise 2.1.2]{H}, we get the following commutative diyagram
 $$\begin{CD}
         &&&&&&&&\\
         \ \ &&&& 0 @>>> ((\mathfrak{F}\mid_U)^*)_{\fp} @>>> ((\overset{t_1}\oplus\mathcal{O}_X\mid_U)^*)_{\fp} @>>> (\overline{\lambda} (\mathfrak{F}\mid_U))_{\fp} @>>>0&  \\
          &&&&&&  @VV{\cong}V   @VV{\cong}V \\
         \ \  &&&& 0 @>>> (\mathcal{H}om_{\mathcal{O}_X}(\mathfrak{F} ,\mathcal{O}_X))_{\fp} @>>> (\mathcal{H}om_{\mathcal{O}_X}(\overset{t_1}\oplus\mathcal{O}_X ,\mathcal{O}_X))_{\fp}
 @>>> (\overline{\lambda} \mathfrak{F})_{\fp} @>>>0.&\\
        \end{CD}$$\\
 This implies $(\overline{\lambda} (\mathfrak{F}\mid_U))_{\fp} \cong (\overline{\lambda} \mathfrak{F})_{\fp}.$ On the other hand,
$ (\overline{\lambda} \mathfrak{F})_{\fp}\cong  ((\overline{\lambda} \mathfrak{F})\mid_U)_{\fp}$ which ends $((\overline{\lambda}\mathfrak{F})\mid _U)_{\fp} \cong (\overline{\lambda}(\mathfrak{F}\mid_U))_{\fp}.$ 

Using similar argument as used above, the second result has desired.
%Similarly, the second result follows from the next diagram and five Lemma.
% $$\begin{CD}
   %      &&&&&&&&\\
      %   \ \ &&&& 0 @>>> \overline{\lambda}( \mathfrak{F}\mid_U) @>>> \mathcal{H}om_{\mathcal{O}_X\mid_U}(\overset{t_2}\oplus\mathcal{O}_X\mid_U ,\mathcal{O}_X\mid_U) @>>> \overline{\Tr} (\mathfrak{F}\mid_U) @>>>0&  \\
         % &&&&&&  @VV{\cong}V   @VV{\cong}V \\
%         \ \  &&&& 0 @>>> (\overline{\lambda}\mathfrak{F})\mid_U @>>> (\mathcal{H}om_{\mathcal{O}_X}(\overset{t_2}\oplus\mathcal{O}_X,\mathcal{O}_X))\mid_U
 %@>>> (\overline{\Tr} \mathfrak{F})\mid_U @>>>0.&\\
   %     \end{CD}$$\\ 
  \end{proof}
% shows that, $\overline{\Tr}$ and $\overline{\lambda}$ are functor on the category of sheaves of modules over an scheme.

Next theorem shows that, in the category of coherent sheaves of modules over a scheme, free resolutions of a coherent sheaf are stability isomorphism and  therefore $\overline{\Tr}$ and $\overline{\lambda},$ under stabilization, are functors.
  \begin{thm}
Let $\mathfrak{F}$ and $\mathfrak{G}$ be coherent $\mathcal{O}_X$-modules such that $\underline{\mathfrak{F}}\cong \underline{\mathfrak{G}}.$ Then   
  $\underline{\overline{\Tr}\mathfrak{F}}\cong \underline{\overline{\Tr}\mathfrak{G}} $
and $\underline{\overline{\lambda}\mathfrak{F}}\cong \underline{\overline{\lambda}\mathfrak{G}}.$
  \end{thm}
  \begin{proof}
  Let $U\subseteq X$ be an arbitrary affine subset. In view of the assumption, there are free sheaves of finite ranks $\mathfrak{F'}$ and $\mathfrak{G'}$ with 
  $\mathfrak{F} \oplus \mathfrak{F'} \cong \mathfrak{G'}\oplus \mathfrak{G'}.$  This, in conjunction
with Remark \ref{B21}(3), implies the following isomorphisms  
 \begin{eqnarray*}
(\overline{\Tr}\mathfrak{G})\mid_U&=& (\overline{\Tr}\mathfrak{G})\mid_U\oplus (\overline{\Tr}\mathfrak{G'})\mid_U \cong(\overline{\Tr}\mathfrak{G}\oplus\overline{\Tr}\mathfrak{G'})\mid_U
\\ &\cong&\underline{\overline{\Tr}(\mathfrak{G}\oplus\mathfrak{G'})}\mid_U\cong\underline{\overline{\Tr}(\mathfrak{F}\oplus\mathfrak{F'})}\mid_U \\&\cong&\underline{(\overline{\Tr}\mathfrak{F}\oplus\overline{\Tr}\mathfrak{F'})}\mid_U\cong \underline{(\overline{\Tr}\mathfrak{F})}\mid_U.
\end{eqnarray*}
Now, the result follows from \ref{R.2} and the fact that 
$\overline{\lambda}\mathfrak{F}\cong \Omega\overline{\Tr}\mathfrak{F}.$
  \end{proof}
  
 The following proposition considers a case where $\overline{\Tr}(\overline{\Tr}\mathfrak{F})=\mathfrak{F}.$

\begin{prop}\label{B6}
Let $\mathfrak{F}$ be a coherent $\mathcal{O}_X$-modules. Then 
$\underline{\overline{\Tr}(\overline{\Tr}\mathfrak{F})}=\underline{\mathfrak{F}}$.
\end{prop}
\begin{proof}
One may assume that $\mathfrak{F}$ is stable. By the assumption,
$X$ can be covered by open affine subsets $ U_i = \Spec R_i$ such that, for each $i$, there exists a finitely generated $R_i$-module $M_i$ with $\mathfrak{F}\mid_{U_i}=\overset{\sim} M_i.$ In view of \ref{B3},
$(\overline{\Tr}(\overline{\Tr}\mathfrak{F}))\mid_{U_i}\cong\overline{\Tr}(\overline{\Tr}(\mathfrak{F}\mid_{U_i})) \cong \mathfrak{F}\mid_{U_i}.$ 
%Set  $\mathfrak{F'}_i=\overline{\Tr}(\overline{\Tr}(\mathfrak{F}\mid_{U_i}))$. 
The result follows from \ref{R.2}.
 
   % $\mathfrak{F'}_i\mid_{U_i\cap U_j}=\mathfrak{F'}_j\mid_{U_i\cap U_j}  $
%  ، بنابر قضیه چسب بافه ها، یک بافه یکتا $ \mathfrak{G} $وجود دارد بطوری که 
% $ \mathfrak{F'}_i=\mathfrak{G}\mid_ {U_i}$. از یکتایی این بافه و این که بافه های   $\mathfrak{F}$ و %$\overline{\Tr}(\overline{\Tr}\mathfrak{F})$ نیز در این شرط صدق می کنند، مشخص می شود که %$\mathfrak{F}=\overline{\Tr}(\overline{\Tr}\mathfrak{F})$.
\end{proof}
By \cite[2.6]{A}, a finitely generated $R$-module $M$ is projective if and only if $\underline{\Tr M}=0$. Inspired by this theorem, we define a $\mathcal{O}_X$-projective sheaf.

\begin{defn}\label{B22}
Let $\mathfrak{F}$ be a $\mathcal{O}_X$-module. We say $\mathfrak{F}$ is a $\mathcal{O}_X$-projective if $\mathfrak{F}$ is direct summand of a free $\mathcal{O}_X$-module with finite rank. In other words, there exist a sheaf of $\mathcal{O}_X$-module $\mathfrak{h}$ and a free $\mathcal{O}_X$-module $\mathfrak{G}$ with finite rank such that $\mathfrak{G}= \mathfrak{F}\oplus \mathfrak{h}$. 
\end{defn}
The following proposition shows the relation between the projectiveness a finitely generated $R$-module and the $\mathcal{O}_X$-projectiveness of its associated sheaf.

\begin{prop}\label{B23}
 Let $X=\Spec(R)$ be an affine scheme and $M$ be a finitely generated $R$-module. Then $M$ is projective if and only if $\overset{\sim}M$ is $\mathcal{O}_X$-projective.
\end{prop}

%\begin{proof}
%Let $M$ be projective. So there are $R$-module $N$ and $l\in \N$ such that $\overset{l}\oplus R = M\oplus N.$ Hence, the result follows as $\overset{l}\oplus \mathcal{O}_X = \overset{\sim}{M} \oplus \overset{\sim}{N}.$ 

%Conversely, assume that there exist $\mathcal{O}_X$-module $\mathfrak{h}$ and $t\in \N$ such that $\overset{l}\oplus \mathcal{O}_X = \overset{\sim}{M} \oplus \mathfrak{h}.$ Applying $\Gamma (X,-)$, we get the result.

%\end{proof}
In the rest of this section, we classify locally free sheaves in terms of their transpose. 

\begin{prop}\label{B4}
Let $\mathfrak{F}$ be a coherent $\mathcal{O}_X$-module. Then 
 the following statments are equivalent.
 \begin{itemize}
\item [(1)] $\mathfrak{F}$ is locally free.
\item [(2)] $\underline{\overline{\Tr}\mathfrak{F}}=0$.
\item [(3)]  $\mathfrak{F}$ is $\mathcal{O}_X$-projective.
\end{itemize}
\end{prop}

\begin{proof}
Note that we may consider the case where $\mathfrak{F}$ is stable.
\begin{itemize}
\item["$1\rightarrow 2$"]
Let $U$ be an arbitrary open affine subset. So, by the assumption and \ref{B21} and also \ref{B5}, 
$ \overline{Tr}(\mathfrak{F}\mid_U)= 0. $
Hence, $ \overline{Tr}\mathfrak{F} =0 .$
\item ["$2 \rightarrow 3$"]  
By the assumption,
$X$ can be covered by open affine subsets $ U_i = \Spec R_i$ such that, for each $i$, there exists a finitely generated $R_i$-module $M_i$ with $\mathfrak{F}\mid_{U_i}=\overset{\sim} M_i.$ 
%So, $0=(\overline{Tr}\mathfrak{F})\mid_{U_i} = \overline{Tr}(\mathfrak{F}\mid_{U_i}) =\overline{Tr}(\overset {\sim}M) = \overset {\sim}{\Tr M}.$ 
Hence, via \ref{B3}, one may see $ \Tr M_i=0.$ So, \cite[2.6]{A} implies
$M_i$ is projective. Therefore, there exist $ \mathcal{O}(U_i) $-module $N_i$ and an integer number $l_i$ such that $\overset{l_i}\oplus \mathcal{O}(U_i) = M_i\oplus N_i.$ Choosing $l:= max \{l_i\}$ (and suitable $N_i$), we can assume that $\overset{l}\oplus \mathcal{O}(U_i) = M_i\oplus N_i$ for all $i.$ Applying "$\sim$,"we get $\overset{l}\oplus \mathcal{O}_X\mid_{U_i} =\mathfrak{F}\mid_{U_i} \oplus \overset{\sim}{N_i}$ and $ \mathfrak{F}\mid_{U_i} $ is $ \mathcal{O}_X\mid_{U_i} $-projective. Setting $ \mathfrak{G}_i=\overset{\sim}{N_i},$ for $i \neq j$, $\overset{l}\oplus \mathcal{O}_X\mid_{U_i\cap U_j} =\mathfrak{F}\mid_{U_i\cap U_j}\oplus \mathfrak{G}_i\mid_{U_i\cap U_j}=\mathfrak{F}\mid_{U_i\cap U_j}\oplus \mathfrak{G}_j\mid_{U_i\cap U_j}.$ Hence,
$ \mathfrak{G}_i\mid_{U_i\cap U_j}\cong\mathfrak{G}_j\mid_{U_i\cap U_j}. $
Using glueing of sheaves lemma, there exists a uniqe sheaf of modules $ \mathfrak{G} $ such that $ \mathfrak{G}_i=\mathfrak{G}\mid_ {U_i}.$ So, $(\overset{l}\oplus \mathcal{O}_X)\mid_{U_i} =\mathfrak{F}\mid_{U_i}\oplus \mathfrak{G}\mid_{U_i}=(\mathfrak{F}\oplus \mathfrak{G})\mid_{U_i}$ and the result follows from \ref{R.2}.

\item ["$3 \rightarrow 1$"]

Let $\fp$ be a point of $X$ and let $\mathfrak{F}$ be a $\mathcal{O}_X$-projective . So, one may find an open affine subset 
 $U\subseteq X$ and $\mathcal{O}(U)$-module $M$ such that $\fp \in U$ and $\mathfrak{F}\mid_U=\overset {\sim}M.$ So, by the assumption, there exist a sheaf of $\mathcal{O}_X$-module $\mathfrak{h}$ and an integer number $l$ such that $\overset{l}\oplus \mathcal{O}_X = \mathfrak{F} \oplus \mathfrak{h}.$ 
%Therefore, $$(\overset{l}\oplus \mathcal{O}_X \mid_U)_{\fp}= (\mathfrak{F}\mid_U)_{\fp} \oplus (\mathfrak{h}\mid_U)_{\fp} =  M_{\fp} \oplus (\mathfrak{h}\mid_U)_{\fp}.$$ So,
This, in conjunction
with Proposition \ref{B23}, implies that $M_{\fp}$ is free. Now,
the result follows from \cite[exercise 2.5.3] {H}.
 \end{itemize}
\end{proof}
%++++++++++++++++++++++++++++++++++++++++++++++++++++++++++++++++++++++++++++++++++++++++++++++++++++
\section{linkage of sheaves of modules}

%-----------------------------------------------------------------------------------------------
%Thorough out this section, we suppose that $X$ is a connected scheme and all sheaves of $\mathcal{O}_X$-module $\mathfrak{F}$ 
%has the ranks $(t_1 , t_2),$ respectively. So $\mathfrak{F}$
%has a free resolution of finite ranks.
 %like $$\overset{t_2}\oplus\mathcal{O}_X  \overset{\phi}{\rightarrow} \overset{t_1}\oplus\mathcal{O}_X \overset{\varphi}{\rightarrow}\mathfrak{F} \rightarrow 0$$
%with $t_1, t_2 \in \N.$
%Therefore, $\mathfrak{F}(U)$ is a finitely generated $\mathcal{O}(U)$-module for every open subset $U\subseteq X.$
In this section, first, we introduce the concept of the linkage of sheaves of modules and
study some basic properties of these sheaves.
Then, using
these properties, we provide that the sheaf made of Glueing schemes and Glueing linked sheaves
of modules is linked.
%The goal of this section is to  and study some of its basic properties.
\begin{defn}\label{B.2}
Let $\mathfrak{F}$ and $\mathfrak{G}$ be two sheaves.
% of $\mathcal{O}_X$-module. 
We say that $\mathfrak{F}$ and $\mathfrak{G}$ are linked 
and denote by $\mathfrak{F}\sim \mathfrak{G}$
if $\mathfrak{F} \cong \overline{\lambda} \mathfrak{G}$ and $\mathfrak{G}\cong \overline{\lambda} \mathfrak{F}.$ Also, $\mathfrak{F}$ is a linked sheaf of modules if there exists a sheaf of $\mathcal{O}_X$-module $\mathfrak{G}$ such that $\mathfrak{F}\sim \mathfrak{G}.$ So, $( \overline{\lambda})^2  \mathfrak{F}\cong \mathfrak{F}.$
\end{defn}
 In view of the definition, one may see that in the case where $\mathfrak{F}$ is a linked sheaf then $\overline{\lambda}\mathfrak{F}$ is so. Indeed, $\overline{\lambda}\mathfrak{F}\sim \mathfrak{F}.$
 Moreover, if $\mathfrak{F}$ is a coherent sheaf then
$\overline{\lambda}\mathfrak{F}$ is a linked coherent sheaf. 

The following Theorem considers the case where $X=\Spec(R)$ is an affine scheme and shows linkedness of a $R$-module
implies linkedness of its associated sheaf and vice versa.
\begin{thm}\label{C1}
 Let $X=\Spec(R)$ be an affine scheme and $M$ and $N$ be 
 % finitely generated 
 $R$-modules. Then the following statments are equivalent.
\begin{enumerate}
\item 
 $ M \sim N $
  \item
 $ \overset {\sim} M \sim  \overset {\sim} N$
  \item
 $\Gamma (U,\overset {\sim} M) \sim  \Gamma (U,\overset {\sim} N),$ for every open subset $U\subseteq X$
  \item
 $\Gamma (X,\overset {\sim} M) \sim  \Gamma (X,\overset {\sim} N)$.
\end{enumerate}
 \end{thm}
\begin{proof}
%\begin{itemize}
%\item[
$ "1\rightarrow 2".$ First note that, by \cite[page 593, proposition3]{MS}, $M$ and $N$ are stable. In view of \ref{B5}, it
is straight forward to see that
$ \overset{\sim}N\cong\overline{\lambda}\overset{\sim} M $ and 
$ \overset{\sim}M\cong\overline{\lambda}\overset{\sim} N.$

$ "2\rightarrow 3".$ The assumption and \ref{B5} imply
  $ \overset{\sim}N\mid_U\cong(\overline{\lambda}\overset{\sim} M)\mid_U\cong\overline{\lambda}(\overset{\sim} M\mid_U)$ 
  and
  $ \overset{\sim}M\mid_U\cong\overline{\lambda}(\overset{\sim} N\mid_U) .$ So, again using \ref{B5}, the result has desired.
%  So, the result again follows by .  
%مجددا بنابر ، $\Gamma (U,\overset {\sim} M\mid_U) = \lambda (\overset {\sim} N(U))$
% و مشابها" $\Gamma (U,\overset {\sim} N\mid_U) = \lambda (\overset {\sim} M(U))$. لذا 
%  $\Gamma (U,\overset {\sim} M) \sim  \Gamma (U,\overset {\sim} N)$.
 
% $ "4\rightarrow 1".$ It follows from the fact that $\Gamma (X,\overset {\sim} M)=M.$
 %\item [ "(3) به  (4)".]
%کافی است قرار دهیم $U:=X$. 
%\item [ "(4) به  (1)".]
%کافی است توجه کنیم که بنابر ،  .
%\end{itemize}
\end{proof}
\begin{cor}\label{D1}
Let $X$ be an affine scheme and $\mathfrak{F}$ and $\mathfrak{G}$ be coherent sheaves. Then 
$\mathfrak{F} \sim \mathfrak{G}$
%$\mathfrak{F}$ and $\mathfrak{G}$ are linked 
if and only if $\Gamma (X,\mathfrak{F})\sim\Gamma (X,\mathfrak{G} ).$ Moreover, $\mathfrak{F} \sim \mathfrak{G}$ if and only if $\overline{\lambda}\mathfrak{F} \cong  \mathfrak{G} $.
\end{cor}
Considering the ideal case, one may obtain another corollary of Theorem \ref{C1}, as follows.
\begin{cor}\label{A3}
%2.2
Let $\fa$ and $ \fb$ be two ideals such that $\fa$ and $ \fb$ are linked by zero ideal. Then, by \cite[page 592 proposition 1]{MS} and \ref{C1}, $ \mathcal{O}_{\frac{R}{\fa}}\sim  \mathcal{O}_{\frac{R}{\fb}}.$
\end{cor}
It is well-known that linked modules are syzygy. To better study the linked sheaves, let's define a syzygy sheaf.
\begin{defn}
We say a sheaf of $\mathcal{O}_X$-module $\mathfrak{F}$ is a syzygy if it can be embeded in a free sheaf of finite rank.
%یک بافه -مدولی  سی زی جی است اگر تکریختی
% $0\rightarrow\mathfrak{F}\rightarrow \overset{l}\oplus\mathcal{O}_X  $
%وجود داشته باشد.
\end{defn}
The next three items consider the question "whether properties of the ring affect the linkedness a module and the linkedness a sheaf of modules?". 
%We will come back to this question again in Section 3.
\begin{cor}\label{C5}
Let $\mathfrak{F}$ be a linked sheaf of modules. So, $\mathfrak{F}\cong \overline{\lambda}(\overline{\lambda}\mathfrak{F}).$ This implies that $\mathfrak{F}$ is a syzygy. In particular, for all open subset $U\subseteq X,$ $\mathfrak{F}(U)$ is a syzygy and so is torsionless.
\end{cor}
\begin{exam}\label{A.3}
Let $ R $ be a PID and $X=\Spec(R)$; Then there is no a linked coherent sheaf of $ \mathcal{O}_X $-module. Indeed, if $\mathfrak{F}$ is linked then, by \ref{C5}, $\mathfrak{F}(U)$ is a syzygy, for each $U\subseteq X$. This implies that $\mathfrak{F}(U)$ is a free module. Therefore, $\mathfrak{F}(U)$ can not be a linked module which is a contradiction with \ref{C1}. 
% زیرا اگر $M$ یک $R$   -مدول پیوندی باشد آنگاه بنابر 
% \cite{MS}[ص 600 نتیجه 6]
% $M$
% یک اما یک مدول آزاد، بنابر تعریف، نمی تواند مدول پیوندی باشد. لذا هیچ 
\end{exam}
\begin{exam}
Let $X=\Spec R$ be an integral affine and $\mathfrak{F}$ be a linked coherent sheaf. Assume that there exist a global section $s \in \Gamma(X, \mathfrak{F})$ and a element $f \in R$ such that $s\mid_{D(f)}=0.$ Then $D(f)=\emptyset.$
\end{exam}
\begin{proof}
In view of \ref{C5}, there exists an injective morphism
 $0\rightarrow\mathfrak{F}\overset{\phi}\rightarrow \overset{l}\oplus\mathcal{O}_X .$ Hence, we get the following commutative digram
$$\begin{CD}
         &&&&&&&&\\
         \ \ &&&& 0 @>>> \mathfrak{F}(X) @>{\phi _X}>>\overset{l}\oplus \mathcal{O}(X)&  \\
          &&&&&&  @VV{\rho_1}V   @VV{\rho_2}V \\
         \ \  &&&& 0 @>>>  \mathfrak{F}(D(f)) @>{\phi _f}>> \overset{l}\oplus \mathcal{O}(D(f)).
 &\\
        \end{CD}$$\\
   One may see that 
        $0= \phi_f\rho_1(s)=\rho_2\phi_X(s)=\phi_X(s)\mid_{D(f)}.$
  Lemma \cite[2.5.3]{H} implies that $f^n\phi_X(s)=0$ for some integer numbers $n.$ Assume that $\phi_X(s)=\sum^l_{i=1}r_ie_i.$ Now, it is straight forward to see that
      %So, $f^nr_i=0$ for all $i.$ As $s\neq 0$ and $\phi_X$ is one to one, $\phi_X(s)\neq 0$ and therefore 
      there exists $r_j\neq 0$ with $f^nr_j=0.$ This implies that $f=0$ and the result has desired.
\end{proof}
%\begin{exam}
%Let $\mathfrak{F}$ be a linked sheaf. Then, for all open subset $U\subseteq X,$ $\mathfrak{F}(U)$ is a torsionless $\mathcal{O}_X(U)$-module. Indeed, by $ 0\rightarrow \mathfrak{F}\rightarrow %\overset{l}\oplus\mathcal{O}_X,$ $\mathfrak{F}(U)$ is a syzygy and so is torsionless.
%\end{exam}
 The next theorem shows that the linkedness of a sheaf is a locally property. 
\begin{thm}\label{C2.3}
Then the following statments hold.
\begin{enumerate}
\item[ (1) ] 
$\mathfrak{F}$ is linked if and only if $\mathfrak{F}\mid_U$ is linked as $\mathcal{O}_X\mid_U$-modules, for every open subset $U\subseteq X.$
\item[ (2) ]
 In addition, in the case where $\mathfrak{F}$ is a coherent sheaf, $\mathfrak{F}$ is linked if and only if $\mathfrak{F}\mid_U$ is linked,
 %as $\mathcal{O}_X\mid_U$-modules, 
 for all open affine subset $U\subseteq X.$
\end{enumerate}
\end{thm}
\begin{proof}
$ (1) $ Let $\mathfrak{F}$ be linked and $U\subseteq X$ be an open subset. Theorem \ref{B5} implies the isomorphisms
$ \mathfrak{F}\mid_U= (\overline{\lambda}^2  \mathfrak{F})\mid_U= \overline{\lambda}((\overline{\lambda}  \mathfrak{F})\mid_U)= \overline{\lambda} (\overline{\lambda} (\mathfrak{F}\mid_U))$ and the result has desired.
The converse follows using similar argument as used above.
% again by , $ \mathfrak{F}\mid_U= \overline{\lambda} \overline{\lambda} ( \mathfrak{F}\mid_U)= \overline{\lambda}((\overline{\lambda}  \mathfrak{F})\mid_U)= (\overline{\lambda}^2  \mathfrak{F})\mid_U$. So, $ \mathfrak{F}= (\overline{\lambda}^2  \mathfrak{F}).$
  
  $ (2) $ Assume that, for all open affine subset $U\subseteq X,$ $\mathfrak{F}\mid_U$ is a linked sheaf of  $\mathcal{O}_X\mid_U$-module. This implies $ \mathfrak{F}\mid_{U_i}= (\overline{\lambda} (\overline{\lambda}\mathfrak{F}))\mid_{U_i}.$ Now, the result follows from \ref{R.2}.
  
  %Set $\mathfrak{G}_i=(\overline{\lambda} \overline{\lambda}\mathfrak{F})\mid_{U_i}.$ So, for $i\neq j,$ 
%$\mathfrak{G}_i\mid_{U_i\cap U_j}=(\overline{\lambda} \overline{\lambda}\mathfrak{F})\mid_{U_i\cap U_j}=\mathfrak{G}_j\mid_{U_i\cap U_j}  $
%Hence, by Glueing of sheaves, there exists a uniqe sheaf of modules $ \mathfrak{G} $ such that $ \mathfrak{G}_i\cong\mathfrak{G}\mid_ {U_i}.$ By uniqness, $\mathfrak{F}\cong\overline{\lambda}( \overline{\lambda}\mathfrak{F}).$
%\end{itemize}
 \end{proof}
 %\begin{rem}\label{r.2} 
%We recall that a finitely generated $R$-module $M$ is linked if and only if $M$ is stable and syzygy (see 
%\end{rem}
The next corollary states an equal condition for linkedness of a sheaf of modules on a scheme. \begin{cor}\label{D4}
Let $\mathfrak{F}$ be a coherent sheaf of $\mathcal{O}_X$-module. Then $\mathfrak{F}$ is linked if and only if $\mathfrak{F}$ is stable and is a syzygy.
 \end{cor}
\begin{proof}
By the previous theorem, it is enough to consider the case where $X=\Spec R$. So, there exists a finitely generated $R$-module with $\mathfrak{F}=\overset {\sim} M.$ Now, using \ref{D1} and \cite[page 600, corollary 6]{MS}, the result has desired . 
 \end{proof}  
 Is a sheaf of modules obtained by glueing of linked sheaves a linked sheaf? To answer this question, it is necessary to state the following concept.
 \begin{defn}
 Let $\mathfrak{F}$ and $\mathfrak{G}$ have the same $(t_1 , t_2)$-ranks. Then we say $\mathfrak{F}$ and $\mathfrak{G}$ are $(t_1 , t_2)$ co-rank.
  \end{defn}
 \begin{thm}\label{C4}
 Let $\{U_i\}$ be a family of open affine subsets $U_i\subseteq X$ and $t$ and $k$ be integer numbers. Assume that 
 $\{(U_i, \mathfrak{F}_i)\}$ is a family of $(t , k)$ co-rank linked coherent sheaves of modules which is true in Glueing lemma. Then there exists a unique linked sheaf $\mathfrak{F}$ on $X$
  such that $ \mathfrak{F}\mid_{U_i}\cong\mathfrak{F}_i ,$ for any $i.$
\end{thm}
 \begin{proof}
 By Glueing lemma, there exists a unique sheaf $\mathfrak{F}$ on $X$
  such that $ \mathfrak{F}\mid_{U_i}\cong\mathfrak{F}_i ,$ for any $i.$ Also, due to $ \mathfrak{F}\mid_{U_i}=\overset{\sim}{\Gamma(U,\mathfrak{F}_i)},$ $\mathfrak{F}$ is a coherent sheaf and has the same $(t , k)$-rank. On the other hand, via the definition and \ref{B5},
  $\mathfrak{F}\mid_{U_i}\cong\mathfrak{F}_i\cong \overline{\lambda}(\overline{\lambda}\mathfrak{F}_i)\cong\overline{\lambda}(\overline{\lambda}(\mathfrak{F}\mid_{U_i}))\cong(\overline{\lambda}(\overline{\lambda}\mathfrak{F}))\mid_{U_i}.$
 Again, the result follows from \ref{R.2}.
 \end{proof}
 \begin{thm}\label{C}
\textbf{Glueing of linked sheaves.}
Let $(X, \mathfrak{F})$ be a sheaf obtained by glueing the family of sheaves $(X_i, \mathfrak{F}_i)$ under glueing sheaves of modules and glueing schemes conditions. Assume that, for all $i,$ the scheme $X_i$ is connected. Then $\mathfrak{F}$ is linked if there exist integer numbers $t$ and $k$ such that $\{X_i, \mathfrak{F}_i\}$ are $(t , k)$ co-rank linked sheaves of modules.
 \end{thm} 
  \begin{proof}
  First note that, in
conjunction with the structure of the scheme made of glueing schemes, implies $X$ is connected. For instance, consider glueing of $X_1$ and $X_2.$ Using a special equation relation, $X: = \frac{X_1\cup X_2}{\thicksim}$ and the intersection of connected schemes is not empty. This emplies that $X$ is connected. Therefore, the result follows from theorem \ref{C4}.
  \end{proof}
The next item consider the question``whether there exists a linked subsheaf of $\mathfrak{F}$ when $\mathfrak{F}$ is an arbitrary sheaf".
\begin{lem}\label{L1}
Let $U\subseteq X$ be an open subset such that $\mathcal{O}(U)$ is not an integral domain.Then $\mathfrak{F}\mid_U$ has a linked subsheaf if and only if 
$\Ass\mathfrak{F}(U)\cap\Ass\mathcal{O}(U)\neq\varnothing.$
\end{lem} 
\begin{proof}
Assume that there is a point $ \fp \in \Ass\mathfrak{F}(U)\cap\Ass\mathcal{O}(U).$ In view of \cite[2.5]{JS2} and \cite[page 592 proposition 1]{MS}, $\frac{\mathcal{O}(U)}{\fp}$ is a linked submodule of $\mathfrak{F}(U).$ So, \ref{C1} implies that $(\frac{\mathcal{O}(U)}{\fp})^{\sim}$ is a linked subsheaf of $\mathfrak{F}\mid_U.$

Conversely, assume that $\mathfrak{F}\mid_U$ has a linked subsheaf $\mathfrak{F'}.$ Via the definition,  $\Ass\mathfrak{F'}(U)\subseteq\Ass\mathcal{O}(U).$ Now, the result has desired from the fact that $\varnothing\neq\Ass\mathfrak{F'}(U)\subseteq\Ass\mathfrak{F}(U).$
\end{proof}
\begin{cor}
Let $\mathcal{O}\mid_U$ is not an integral domain and $\Ass\mathfrak{F}(U)\cap\Ass\mathcal{O}(U) =\varnothing,$ for all open subset $U\subseteq X.$ Then $\mathfrak{F}$ and $\mathfrak{F}\mid_U$ are not linked and they have not a linked subsheaf.
\end{cor}
In the rest of this section we study whether exists a maximal linked subsheaf.

\begin{thm}\label{p.}
Let $X$ be a Noetherian scheme and $\mathfrak{F}$ be a coherent sheaf. Suppose that there exists an open affine subset $U\subseteq X$ such that $\mathcal{O}\mid_U$ is not an integral domain and $\Ass\mathfrak{F}(U)\cap\Ass\mathcal{O}(U) \neq\varnothing.$ Then the set 
$$\sum = \{ \mathfrak{F'}\mid \mathfrak{F'} \text{ is a linked subsheaf of } \mathfrak{F}\mid_ U \}$$
is not empty and has maximal element. Moreover, for any two maximal members $\mathfrak{F'}$ and $\mathfrak{F''},$ $\mathfrak{F'}\oplus \mathfrak{F''}$ is not stable.
\end{thm}

\begin{proof}
The first part follows from \ref{L1}. Also, via the assumption, there exists a finitely generated $\mathcal{O}(U)$-module with $\mathfrak{F}\mid_U \cong \overset{\sim}M.$ So, in conjuction with \cite[2.5.5]{H} and \ref{C1}, there is an One-to-one correspondence between $\sum$ and 
$$\sum^{'}_U = \{N\leqslant M\mid N  \text{ is a linked submodule of } M \}.$$
 Hence, the existence of the maximal member of the set $\sum^{'}_U $ causes the existence of the other.

In case $\mathfrak{F'}$ and $\mathfrak{F''}$ are maximal members, assume that $\mathfrak{F'}\oplus \mathfrak{F''}$ is stable. Via the definition, $\mathfrak{F'}\oplus \mathfrak{F''}$ is a syzygy and so a linked sheaf which is a contradiction.
\end{proof}
%++++++++++++++++++++++++++++++++++++++++++++++++++++++++++++++++++++++++++++++++++++++++++++++++++++

%\vskip 1 cm
%%%%%%%%%%%%%%%%%%%%%%%%
\section*{Acknowledgement}
The authors would like to thank Iran national science foundation for scientific and financial support of this project.
%%% ----------------------------------------------------------------------
%%% ----------------------------------------------------------------------
%%%
%++++++++++++++++++++++++++++++++++++++++++++++++++++++++++++++++++++++++++++++++++++++++++++++++++++
%%% ----------------------------------------------------------------------
%%% ----------------------------------------------------------------------
%%% ----------------------------------------------------------------------
\bibliographystyle{amsplain}
%%% ----------------------------------------------------------------------
%%% ----------------------------------------------------------------------
%%% ----------------------------------------------------------------------

%$\address{Faculty of Mathematical Sciences and Computer, Amirkabir University, Tehran, Iran.}$
%%% ----------------------------------------------------------------------
%%% ----------------------------------------------------------------------
%%% ----------------------------------------------------------------------
\end{document}